\documentclass[11pt, reqno]{amsart}
\usepackage{amsmath,amsthm,amsfonts,amssymb,mathrsfs,bm,graphicx,stmaryrd}
\usepackage{epstopdf}
\usepackage{mathtools}
\usepackage{dsfont}
\usepackage{multicol}
\usepackage[colorlinks=true,linkcolor=blue,citecolor=blue,breaklinks]{hyperref}
\usepackage{url}

\usepackage{breakurl}
\usepackage{bbm}
\usepackage{subcaption}
\usepackage{enumerate}
\usepackage{longtable}

\newcommand{\scX}{\mathscr{X}}

\newcommand{\QQ}{\mathbb{Q}}
\newcommand{\bz}{\text{\usefont{U}{bboldx}{m}{n}z}}
\newcommand{\bw}{\text{\usefont{U}{bboldx}{m}{n}w}}

\newcommand{\CC}{\mathbb{C}}

\newcommand{\PP}{\mathbb{P}}

\newcommand{\cT}{\mathcal{T}}

\newcommand{\NN}{\mathbb{N}}

\newcommand{\ZZ}{\mathbb{Z}}

\usepackage[letterpaper,hmargin=1.0in,vmargin=1in]{geometry}
\parskip 	\smallskipamount

\newtheorem{theorem}{Theorem}
\newtheorem{lemma}[theorem]{Lemma}

\newtheorem{definition}[theorem]{Definition}

\newtheorem{proposition}[theorem]{Proposition}

\newcommand{\Cov}{\mathrm{Cov}}

\newcommand{\red}{\textcolor{red}}

\usepackage[backend=biber,style=alphabetic,doi=false,maxalphanames=10,maxnames=50]{biblatex}
\AtBeginBibliography{\small}

\begin{document}
\title[]{Uniformity of extremal behaviour along geodesics in Liouville quantum gravity}

\author[]{Manan Bhatia}
\address{Manan Bhatia, Department of Mathematics, The University of Hong Kong, Hong Kong}
\email{mananb@hku.edu}
\author[]{Konstantinos Kavvadias}
\address{Konstantinos Kavvadias, Courant Institute of Mathematical Sciences, New York University, New York, NY, USA}
\email{kk6501@nyu.edu}
\date{}
\maketitle
\begin{abstract}
In random geometry, geodesics often have a tendency to coalesce together and traverse regions highly singular relative to typical environments. In this paper, working with the model of Liouville quantum gravity, we develop a technique which yields zero-one laws for many extremal statistics measured along interior segments of the geodesic. Namely, working with statistics such as the Euclidean dimension, optimal H\"older continuity exponents with respect to the Euclidean metric and the minimal/maximal encountered thickness for the underlying GFF, we obtain zero-one laws for the above for segments in the \emph{bulk} of geodesics, thereby upgrading the results of \cite{GPS20}. The primary technique used, developed in \cite{BK25}, is to lay down a large family of small-scale ``typical'' and well-behaved geodesics along interior segments of a long geodesic.

\end{abstract}

\section{Introduction}
\label{sec:intro}

\subsection{Overview}
Liouville quantum gravity (LQG) was first introduced in 1980s in the physics literature \cite{polyakov1981quantum, Dav88, DK89} as a family of random surfaces arising as scaling limits of planar map models. In the past two decades, LQG has been rigorously constructed as a random measure metric space \cite{DS09,DDDF20,GM21}. For a fixed $\gamma \in (0,2)$, an open and connected set $U \subseteq \mathbb{C}$, and with $h$ being a Gaussian free field (GFF) on $U$, the $\gamma$-\emph{Liouville quantum gravity} (LQG) surface parametrised by $(U,h)$ is formally the random two-dimensional Riemannian manifold with metric tensor 
\begin{equation}\label{eqn:metric_tensor}
 e^{\gamma h} (dx^2 + dy^2),   
\end{equation}
where $dx^2 + dy^2$ denotes the Euclidean metric tensor. Note that the definition \eqref{eqn:metric_tensor} does not make literal sense since $h$ is not well-defined as a function taking pointwise values but rather as a distribution (generalized function). Nevertheless, by a renormalisation approach, it has been established \cite{DDDF20,GM21} that one can indeed construct a (random) metric $D_h$ on $U$ associated with \eqref{eqn:metric_tensor}.

In view of \eqref{eqn:metric_tensor}, $\gamma$-LQG can naturally be viewed as a \emph{random geometry} obtained as a perturbation of the usual Euclidean plane. In the past decade, there has been rapid progress in the understanding of such random geometries, with some other prominent examples being: the directed landscape \cite{DOV18}, first and last passage percolation models \cite{Cor12,ADH17,Gan22}, Brownian geometry \cite{LeGal19}, Kendall's Poisson roads metric \cite{Ken17,BCK25}, critical long range percolation \cite{Bau23,DFH26}, and metrics associated to conformal loop ensembles \cite{MY25,MT25}. All these models, though very different in their details, do possess many commonalities. Owing to the phenomenon of geodesic coalescence, one expects that a geodesic between two ``typical'' points in all of the random geometries, rapidly merges with ``highway'' regions which are certain exceptionally attractive portions of the space.
As a result, one might expect the environment seen along the ``bulk'' of such a geodesic to potentially be very different from that seen around typical regions, and there has recently been significant interest \cite{Die16,BBG24,Mou24,DSV22,MSZ25,Bat24} recently in understanding this disparity.

Returning back to the specific setting of $\gamma$-LQG, in light of the above discussion, if we consider a geodesic $\Gamma_{z,w}$ between points $z\neq w\in \CC$, then one might hope that the environment as one zooms in around a point $\Gamma_{z,w}(t)$ for $t\in (0,D_h(z,w))$ to converge to a scale invariant object with a trivial tail $\sigma$-algebra around it. For instance, such a result has been proved for the Brownian map \cite{Mou24} and also the directed landscape \cite{DSV22}. While we do not establish such a convergence in this paper for $\gamma$-LQG, we develop a technique which allows us to obtain zero-one laws for many extremal statistics measured along the ``bulk'' of a geodesic. Specifically, given a geodesic $\Gamma_{z,w}$ between points $z\neq w\in \CC$ and parametrised according to its arc-length parametrisation $P\colon [0,D_h(z,w)]\rightarrow \CC$, for $[s,t]\subseteq (0,D_h(z,w))$, we consider statistics capturing the extremal behaviour encountered by the path $\Gamma_{z,w}\lvert_{[s,t]}$ during its journey from $\Gamma_{z,w}(s)$ to $\Gamma_{z,w}(t)$. As an example of such an extremal statistic, one might consider the extremal (maximal/minimal) thickness (in the sense of a GFF) encountered by $\Gamma_{z,w}\lvert_{[s,t]}$. In this paper, we develop a technique to establish that such statistics are a.s.\ constant independently of the points $z,w$, the geodesic $\Gamma_{z,w}$ and of the bulk interval $[s,t]$.

The results of this paper upgrade the ones of \cite{GPS20}, which establishes a zero-one law for similar statistics measured along the full-length of geodesics $P$ emanating from a fixed point $z\in \CC$. Namely, by using the tail triviality for the GFF around the point $z$, \cite{GPS20} concludes zero-one laws holding for the \emph{entirety} (as opposed to every sub-segment) of the geodesic $P$. In our case, we leverage a covering argument from \cite{BK25} -- itself motivated by the Brownian geometry work \cite{MQ20} -- to lay down a large family of short typical geodesics along the length of a long geodesic (see Figure \ref{fig:chi-on-geod}). By additionally demanding certain regularity conditions on the structure of these small geodesics, we can establish that corresponding zero-one laws hold everywhere in the \emph{bulk} of the geodesic $P$, rather than only near the ends of geodesics with typical endpoints (see Section \ref{subsec:outline} for a detailed comparison).

\subsection{Main results} Let us now state the main results of the paper. Fix $\gamma \in (0,2)$. We shall work with a whole plane GFF $h$ and the associated $\gamma$-LQG metric $D_h$, where $h$ shall always be normalised to to have average $0$ on the unit circle $\partial B_{1}(0)=\{|z|=1\}$. We refer the reader to Section \ref{sec:preliminaries} for a short introduction to the GFF, its thick points and the associated LQG; for $\alpha\in [-2,2]$ we shall use $\cT_h^\alpha\subseteq \CC$ to denote the set of $\alpha$-thick points of the GFF $h$. For a $D_h$-geodesic $P : [0,T] \to \CC$ parametrised to cover unit $D_h$-length in unit time, we set
\begin{align*}
  \widetilde{\chi}(P)&=\sup\{\chi: \exists C>0 \textrm{ such that } |P(s)-P(t)|\leq C|s-t|^{\chi} \textrm{ for all } s,t \in [0,T]\},\nonumber\\
  \widetilde{\chi}'(P)&=\inf\{\chi': \exists C>0 \textrm{ such that } |P(s)-P(t)|\geq C|s-t|^{\chi'} \textrm{ for all } s,t \in  [0,T]\},
\end{align*}
along with
\begin{align}
  \label{eq:14}
  \chi(P)&=\lim_{\varepsilon\rightarrow 0}\widetilde{\chi}(P\lvert_{[\varepsilon,T-\varepsilon]}),\nonumber\\
  \chi'(P)&=\lim_{\varepsilon\rightarrow 0}\widetilde{\chi}'(P\lvert_{[\varepsilon,T-\varepsilon]}),
\end{align}
where we note the first limit above is increasing while the second one is decreasing as $\varepsilon\rightarrow 0$. Here, $\chi(P),\chi(P')$ are defined to capture the optimal H\"older continuity exponents in the ``bulk'' of the geodesic, without being affected by the potentially outsized influence of the behaviour near the endpoints $P(0),P(T)$.
We also define
\begin{align*}
    \alpha_{\mathrm{max}}(P):=\sup\{\alpha \in [-2,2] : P\lvert_{(0,T)} \cap \cT_h^{\alpha} \neq \emptyset\},\\
    \alpha_{\mathrm{min}}(P):=\inf\{\alpha \in [-2,2]: P\lvert_{(0,T)} \cap \cT_h^{\alpha} \neq \emptyset\},
\end{align*}
and these capture the extremal thicknesses in the bulk of the geodesic $P$.
We are now ready to state the main result of the paper.
\begin{theorem}
  \label{thm:1}
Fix $\gamma \in (0,2)$ and $\alpha \in [-2,2]$, and let $h$ be a whole-plane GFF. %
Then, there exist deterministic constants $\Delta^{\mathrm{Euc}}, \Delta_{\alpha}^{\mathrm{Euc}}, \Delta_{\alpha}^{\mathrm{LQG}}, \chi, \chi', \alpha_{\mathrm{min}}$, and $\alpha_{\mathrm{max}}$, such that the following holds almost surely. For any $D_h$-geodesic $P\colon [0,T]\rightarrow \CC$ parametrised by $D_h$-length and any $[s,t]\subseteq (0,T)$, we have
  \begin{enumerate}
  \item \label{it:eucdim} $\dim_{\mathrm{Euc}}(P\lvert_{[s,t]})=\Delta^{\mathrm{Euc}}$.
  \item $\dim_{\mathrm{Euc}}(P\lvert_{[s,t]}\cap \cT_h^{\alpha})= \Delta^{\mathrm{Euc}}_{\alpha}$ and $\dim_{\mathrm{LQG}}(P\lvert_{[s,t]}\cap \cT^\alpha_h)= \Delta^{\mathrm{LQG}}_{\alpha}$.
  \item \label{it:holder} $\chi(P\lvert_{[s,t]})=\chi$ and $\chi'(P\lvert_{[s,t]})=\chi'$.
  \item \label{it:thickness} $\alpha_{\mathrm{max}}(P|_{[s,t]})=\alpha_{\mathrm{max}}$ and $\alpha_{\mathrm{min}}(P|_{[s,t]})=\alpha_{\mathrm{min}}$.
 \end{enumerate}
\end{theorem}

\begin{figure}[]
    \centering
    \begin{subfigure}{0.4\textwidth}
        \centering
        \includegraphics[width=\textwidth]{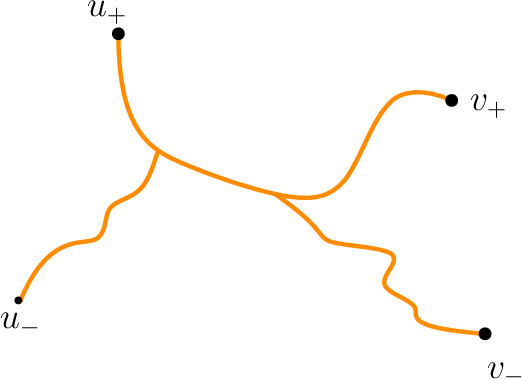}
    \end{subfigure}
    \hfill
    \begin{subfigure}{0.4\textwidth}
        \centering
        \includegraphics[width=\textwidth]{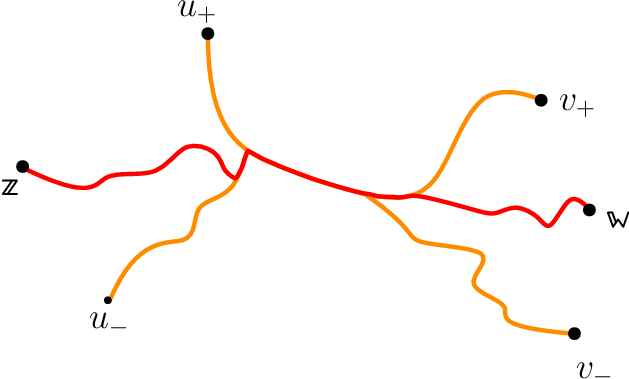}
    \end{subfigure}
    \caption{The left panel depicts the occurrence of the event $\scX=\scX^{u_+,v_+}_{u_-,v_-}$. The unique $D_h$-geodesic from $u_+$ to $v_+$ merges with the unique $D_h$-geodesic from $u_-$ to $v_-$ and after they merge,  the geodesics overlap on a segment up until they separate. After the separation point, the geodesics do not intersect again.  In the right panel, the $D_h$-geodesic $\Gamma_{\bz,\bw}$ (red) passes through the $\scX$ in the sense that $\Gamma_{u_+,v_+}\cap \Gamma_{u_-,v_-}\subseteq \Gamma_{\bz,\bw}$ with $\Gamma_{u_+,v_+}$ and $\Gamma_{u_-,v_-}$ being on opposite sides of $\Gamma_{\bz,\bw}$. This original source of this figure is \cite[Figure 2]{BK25}. %
    }
    \label{fig:two-panels}
\end{figure}

\subsection{Context and proof outline}
\label{subsec:outline}

\begin{figure}
  \centering
  \includegraphics[width=0.8\linewidth]{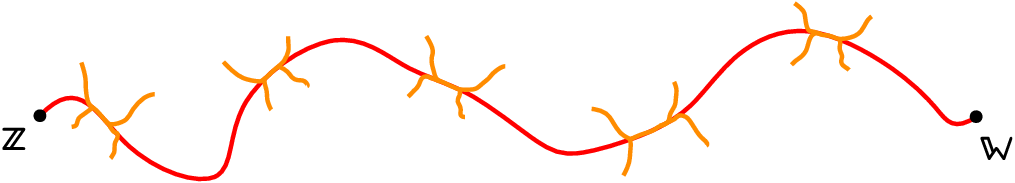}
  \caption{The geodesic $\Gamma_{\bz,\bw}$ passing through multiple small $\scX$s. The original source of this figure is \cite[Figure 1]{BK25}.}
  \label{fig:chi-on-geod}
\end{figure}

We now compare the results of this paper with those from \cite{GPS20} and give a broad overview of the strategy of the proof of Theorem~\ref{thm:1}. For clarity of exposition, we focus on item \eqref{it:eucdim} in Theorem \ref{thm:1}, namely zero-one laws for the Euclidean dimension of geodesic segments.

First, we discuss the corresponding zero-one law for the Euclidean dimension proved in \cite{GPS20}. By leveraging the tail triviality of the GFF around a fixed point $z$ and leveraging geodesic coalescence, \cite{GPS20}, in particular, constructs a semi-infinite\footnote{An infinite path starting from $z$ whose every segment is a geodesic between its endpoints.} geodesic $P_z^\infty$ from $z$ to $\infty$ and show that almost surely, there exists a constant $\Delta^{\mathrm{Euc}}$ such that for any $\varepsilon>0$, we have $\dim_{\mathrm{Euc}}(P^\infty_z\lvert_{[0,\varepsilon]})=\Delta^{\mathrm{Euc}}$. We now give a heuristic overview of the construction of $P_z^\infty$ and above argument, presented somewhat differently from the original source. For a large but fixed $A>1$, let $E_r(z)$ be the event that all geodesics from points in $\partial B_r(z)$ to points in $\partial B_{Ar}(z)$ pass through a common (random) point $p\in B_{Ar}(z)\setminus B_r(z)$ and that all such geodesics are determined by $\sigma(h\lvert_{B_{A^2r}(z)\setminus B_{A^{-1}r}(z)})$. Then \cite{GPS20} establishes that $\PP(E_r(z))>0$ independently of $r$ (Lemma 4.8 therein) as long as $A$ is chosen to be large enough. As a consequence, by the independence structure of the GFF when zooming into the point $z$, one can prove that almost surely, the events $E_{r_i}(z)$ occur for a bi-infinite sequence of increasing dyadic radii $\{r_i\}_{i\in \ZZ}\subseteq \{2^n\}_{n\in \ZZ}$ with the corresponding coalescence points being denoted by $\{p_i\}_{i\in \ZZ}$ (we refer the reader to \cite[Section 2.1]{BBG24} for this formulation). The geodesic $P_z^\infty$ now is defined to be the concatenation $\bigcup_{i\in \ZZ} \Gamma_{p_i,p_{i+1}}$. The fragments $\Gamma_{p_i,p_{i+1}}$ can be seen to have the same law up to scaling with rapid decay of correlations, and as a consequence, by a law of large numbers heuristic, one would expect that almost surely,
\begin{equation}
  \label{eq:5}
  \dim_{\mathrm{Euc}}(P^\infty_z\lvert_{[0,\varepsilon]})= \operatorname{ess\,sup}(\dim_{\mathrm{Euc}}(\Gamma_{p_0,p_{1}}))\footnote{Here, for a random variable $X$, $\operatorname{ess\,sup}(X)$ refers to the supremum of the support of the law of $X$}.
\end{equation}
As discussed in \cite[Remark 4.3]{GPS20}, the above argument does not rule out that $\dim_{\mathrm{Euc}}(P^\infty_z\lvert_{[0,\varepsilon]})$ could be entirely carried by the end point $z$, in the sense that for any $0<s<t<\infty$, it is entirely possible that $\dim_{\mathrm{Euc}}(P^\infty_z\lvert_{[s,t]})$ is random and strictly smaller than $\Delta^{\mathrm{Euc}}$, though it converges almost surely to the value $\Delta^{\mathrm{Euc}}$ as $s\downarrow 0$. The goal of Theorem \ref{thm:1} \eqref{it:eucdim} is to rule out the above, thereby establishing that the Euclidean dimension of every segment in the bulk of a geodesic is deterministic, and therefore equal to $\Delta^{\mathrm{Euc}}$.

In order to achieve this, we will use certain techniques introduced in \cite{BK25}. In particular, motivated by the work \cite{MQ20} in Brownian geometry, \cite{BK25} uses the presence of a specific configuration which we denote by $\scX$. For points $u_+,v_+,u_-,v_- \in \CC$, we say that $\scX_{u_-,v_-}^{u_+,v_+}$ occurs if the scenario described in the first panel of Figure \ref{fig:two-panels} occurs; we shall use $P_+$ (resp.\ $P_-$) to denote the geodesic between $u_+,v_+$ (resp.\ $u_-,v_-$). Moreover, for points $\bz,\bw \in \CC$ and $[s,t]\subseteq (0,D_h(\bz,\bw))$, we say that $\Gamma_{\bz,\bw}\lvert_{[s,t]}$ passes through the $\scX_{u_-,v_-}^{u_+,v_+}$ if the scenario in the second panel holds, that is, if $P_+\cap P_-\subseteq \Gamma_{\bz,\bw}\lvert_{[s,t]}$. It is shown in \cite{BK25} (see \cite[Theorem~2, Proposition~35]{BK25}) that it is almost surely the case that $\Gamma_{\bz,\bw}\lvert_{[s,t]}$ passes through (see Figure \ref{fig:chi-on-geod}) a significant number of $\scX$s corresponding to some points $u_+,v_+,u_-,v_- \in \QQ^2$, where these $\scX$s should be thought of as existing at a very small scale compared to the geodesic $\Gamma_{\bz,\bw}$. Now, by tweaking the parameters in the definition of the $\scX$ above, one can take $u_+,u_-$ (resp.\ $v_+,v_-$) to be very close to each other, and by geodesic coalescence this can be used to ensure that the geodesic $P_+$ from $u_+$ to $v_+$ and $P_-$ from $u_-$ to $v_-$ overlap for most of their journeys. However, since the points $u_+,u_-,v_+,v_-$ are all ``typical''\footnote{More accurately, we use that the points $u_+,u_-$ (resp.\ $v_-,v_+$) are located very close to a typical point, labelled later as $z-r$ (resp.\ $z+r$) in Figure \ref{fig:setup1}. As $\Phi$ therein is taken smaller, with high probability, $P_-\cap P_+$ occupies a progressively larger fraction of the geodesics $P_-,P_+$ and the geodesic $\Gamma_{z-r,z+r}$.}, the results of \cite{GPS20} can be used to ensure that on a positive probability good event depending on $\delta>0$, $\dim_{\mathrm{Euc}}(P_+\cap P_-)> \Delta^{\mathrm{Euc}}-\delta$. Thus, if $\Gamma_{\bz,\bw}\lvert_{[s,t]}$ passes through such a $\scX_{u_-,v_-}^{u_+,v_+}$, then we would in particular have $\dim_{\mathrm{Euc}}(\Gamma_{\bz,\bw}\lvert_{[s,t]})\geq \dim_{\mathrm{Euc}}(P_+\cap P_-)> \Delta^{\mathrm{Euc}}-\delta$, and since $\delta$ is arbitrary, we would obtain $\dim_{\mathrm{Euc}}(\Gamma_{\bz,\bw}\lvert_{[s,t]})\geq \Delta^{\mathrm{Euc}}$. Finally, coming to the upper bound, the results of \cite{BK25} allow us to fully cover the interior of any geodesic by geodesics between rational points, and this combined with the results of \cite{GPS20} can be straightforwardly used to obtain the reverse inequality $\dim_{\mathrm{Euc}}(\Gamma_{\bz,\bw}\lvert_{[s,t]})\leq \Delta^{\mathrm{Euc}}$.

In this exposition, while we focused on the case of the Euclidean dimension of a geodesic, we expect it to be applicable broadly for ``extremal'' statistics, in the sense of obtaining zero-one laws for the extremal behaviour observed as a geodesic travels from $P(s)$ to $P(t)$ for ``bulk'' intervals $[s,t]$. To illustrate this, we note that the same approach applies to the maximal/minimal thickness observed as one travels along a geodesic segment (Theorem \ref{thm:1} \eqref{it:thickness}). On the contrary, this approach would not yield any information for the ``average'' thickness (see \cite[Problem 5.3]{DDG21}) along $\Gamma_{\bz,\bw}\lvert_{[s,t]}$ since the $\scX^{u_+,v_+}_{u_-,v_-}$ events and the corresponding segments $P_-\cap P_+$ a priori cover only a vanishing fraction of the overall length of the geodesic $\Gamma_{\bz,\bw}\lvert_{[s,t]}$.

\subsection{Notation}
For a simple curve $P : [a,b] \to \CC$ viewed as a path from $P(a)$ to $P(b)$, we will denote by $P^{\text{L}}$ (resp.\ $P^{\text{R}}$) the left (resp.\ right) side of $P$, and it is defined as a collection of prime ends of the boundary of the simply connected domain $\CC \cup \{\infty\} \setminus P$. Thus, for all $r \in (a,b)$, we have that $P^{\text{L}}(r) \neq P^{\text{R}}(r)$ and for $r \in \{a,b\}$, we have $P^{\text{L}}(r) = P^{\text{R}}(r)$. For a curve $P$ as above and another curve $P' : [a',b'] \to \CC$ such that $P'(b') \in P((a,b))$ but $P'([a',b')) \cap P([a,b]) = \emptyset$, we have that $P'(b') \in P^{\star}$ for precisely one $\star \in \{\text{L},\text{R}\}$ and we say that $P'$ is to the left (resp.\ right) of $P$ if $\star = \text{L}$ (resp.\ $\star = \text{R}$).

For points $z , w \in \CC$, we will denote by $\Gamma_{z,w}$ a $\gamma$-LQG geodesic from $z$ to $w$, and we emphasize that it is possible to have points $z,w$ with multiple possible choices of $\Gamma_{z,w}$. Furthermore, unless otherwise stated, the $D_h$-geodesics will always be parametrised to cover unit $D_h$-length in unit time. Also, to avoid clutter, we do not notationally distinguish between a geodesic $\Gamma_{z,w}$ and its graph (seen as a subset of $\CC$); for example, we might write $z=\Gamma_{z,w}(0)\in \Gamma_{z,w}$. Finally, for $z \in \mathbb{C}$ and $r>0$, we shall use $B_r(z)$ to denote the Euclidean ball of radius $r$ centered at $z$.

\textbf{Acknowledgements.} K.K. was supported by the Simons Collaboration Grant
\emph{Probabilistic Paths to Quantum Field Theory}.

\section{Preliminaries}
\label{sec:preliminaries}

\subsection{Gaussian free field} 
\label{subsec:gff}

Throughout this paper, we shall work with a whole-plane GFF $h$ which is defined as the centered Gaussian process $h$ with covariances given by
\begin{equation}\label{eqn:gff_covariance}
    \Cov(h(u) , h(v)) = G(u,v) = \log \,\,\frac{\max\{|u|,1\}\max\{|v|,1\}}{|u-v|}
\end{equation}
for all $u,v \in \CC$. Since the covariance kernel $G(u,v)$ explodes along the diagonal, $h$ cannot be well-defined pointwise almost surely. However, it is well-defined as a distribution (generalized function) in the sense that almost surely, for any bump function $\phi$, the average $(h,\phi) = \int h(u) \phi(u) du$ is well-defined. Moreover, it is shown in \cite[Proposition~3.1]{DS09} that if $h_r(z)$ denotes the average of $h$ on $\partial B_r(z)$, there almost surely exists a version of $h$ such that the map $(z,r) \mapsto h_r(z)$ is almost surely continuous. Also, the renormalization in \eqref{eqn:gff_covariance} is chosen so that $h_1(0) = 0$. Furthermore, the law of $h$ is scale, translational and rotationally invariant in the sense that for every fixed $z \in \CC, r>0$, and $\theta \in [0,2\pi)$,  the laws of the fields $h(\red{e^{i\theta}}\cdot),h(\cdot+z) - h_1(z)$ and $h(r \cdot) - h_r(0)$ are the same.

For $\alpha\geq 0$, we say that a point $z\in \CC$ is $\alpha$-thick if $\liminf_{\varepsilon\rightarrow 0}\frac{h_\varepsilon(z)}{\log\varepsilon^{-1}}=\alpha$. For $\alpha<0$, $z$ is $\alpha$-thick for $h$ if it is $-\alpha$-thick for $-h$. Heuristically, thick points describe the regions where $h$ takes atypically large and low values, and it is known that almost surely, with $\cT_h^\alpha\subseteq \CC$ denoting the set of thick points, one has $\dim_{\mathrm{Euc}}(\cT_h^\alpha)=(2-\alpha^2/2)_{+}$ \cite{HMP10}, with the set being empty if $\alpha\notin [-2,2]$.
 
\subsection{The $\gamma$-LQG metric $D_h$} Given a whole plane GFF $h$ and a $\gamma\in (0,2)$, the LQG metric $D_h$ is defined as a metric on $\CC$ measurable with respect to $h$ satisfying a certain axiomatic definition \cite[Theorem~1.2]{GM21} in accordance with the heuristic definition \eqref{eqn:metric_tensor}. In the works \cite{DDDF20,DFGPS20,GM21}, it was shown that there exists a unique metric $D_h$ satisfying the above axioms and this rigorously defines the LQG metric. Since this axiomatic definition is by now standard in the literature, we shall not repeat it here. The space $(\CC,D_h)$ is almost surely geodesic in the sense that with the length $\ell(\eta;D_h)$ of a path $\eta\colon[0,1]\rightarrow \CC$ being defined by $\ell(\eta;D_h)=\sup_{0=a_0<\dots<a_n=1}\sum D_h(\eta(a_i),\eta(a_{i+1}))$, with the supremum being over all partitions of $[0,1]$, there almost surely exists a geodesic $\Gamma_{z,w}$ satisfying $\ell(\Gamma_{z,w};D_h)=D_{h}(z,w)$ simultaneously for all $z,w\in \CC$, and we shall always parametrise such geodesics to cover unit $D_h$-length in unit time. Further, for fixed points $z,w$ there exists \cite{MQ18} a unique such geodesic $\Gamma_{z,w}$. Finally, both $D_h$ and the Euclidean metric are a.s.\ locally H\"older continuous with respect to each other with the optimal H\"older continuity exponents varying in terms of $\gamma$ \cite[Theorem~1.7]{DFGPS20}.

\subsection{Zero-one laws for LQG geodesics}
We now state the main result from \cite{GPS20} concerning zero-one laws for LQG geodesics starting from deterministic points.

\begin{proposition}\label{prop:zero_one_law}
Fix $\gamma \in (0,2)$ and $\alpha \in [-2,2]$, and let $h$ be a whole-plane GFF. Then, there exist deterministic constants $\Delta^{\mathrm{Euc}},\Delta_{\alpha}^{\mathrm{Euc}},\Delta_{\alpha}^{\mathrm{LQG}},\chi,\chi',\alpha_{\mathrm{min}}$, and $\alpha_{\mathrm{max}}$, such that we have the following. Fix $\bz\in \CC$ and let $P: [0,T] \to \CC$ be any $D_h$-geodesic starting from $\bz$ and parametrised to have unit speed with respect to its $D_h$-length. Then the following hold almost surely.
\begin{enumerate}
  \item \label{it:euclidean_deterministic}
  $\dim_{\mathrm{Euc}}(P)=\Delta^{\mathrm{Euc}}$,
  \item \label{it:thick_points_deterministic}
  $\dim_{\mathrm{Euc}}(P\cap \cT_h^{\alpha})= \Delta^{\mathrm{Euc}}_{\alpha}$ and $\dim_{\mathrm{LQG}}(P\cap \cT^\alpha_h)= \Delta^{\mathrm{LQG}}_{\alpha}$,
  \item \label{it:optimal_holder_deterministic}
  $\chi(P)=\chi$ and $\chi(P)=\chi'$,
  \item \label{it:max_min_thickness_deterministic}
  $\alpha_{\mathrm{max}}(P)=\alpha_{\mathrm{max}}$ and $\alpha_{\mathrm{min}}(P)=\alpha_{\mathrm{min}}$.
 \end{enumerate}   
\end{proposition}

\begin{proof}
Items \eqref{it:euclidean_deterministic} and \eqref{it:thick_points_deterministic} follow from results in \cite{GPS20} (see \cite[Theorem~1.8, Remark~1.12]{GPS20}. Items \eqref{it:optimal_holder_deterministic} and \eqref{it:max_min_thickness_deterministic} above were not considered in \cite{GPS20}, but precisely the same zero-one law argument using the tail-triviality of the Gaussian free field when zooming in around the fixed point $\bz$ yields the above result as well.  
\end{proof}

\subsection{Strong confluence of LQG geodesics and consequences}
As mentioned earlier, geodesics for the LQG metric enjoy the property of geodesic confluence, which refers to the underlying tendency of geodesics to merge with each other. Progressively stronger versions of this phenomenon have been established \cite{GM20,GPS20,BK25} and we now state the version from \cite{BK25} which we refer to as \emph{strong confluence}.

\begin{proposition}(\cite[Theorem~2]{BK25})
\label{prop:strong_confluence}
Fix $\gamma \in (0,2)$ and consider a whole-plane GFF $h$ and the associated LQG metric $D_h$. Then the following holds almost surely. For any points $u,v \in \CC$, any $D_h$-geodesic $P$ from $u$ to $v$ and any sequence of $D_h$-geodesics $\{P_n\}_{n \in \NN}$ converging to $P$ in the Hausdorff sense with respect to the Euclidean metric (or equivalently, the metric $D_h$), for all $n$ large enough, we have 
\begin{equation*}
    (P_n \setminus P) \cup (P \setminus P_n) \subseteq B_{\varepsilon}(u) \cup B_{\varepsilon}(v).
\end{equation*}  
\end{proposition}
The above result is often very useful in practice as it allows us to often establish properties holding for all geodesics $P$ by instead choosing a sequence of typical geodesics $P_n$ converging to $P$ and only examining the path properties of these typical geodesics. We also mention a consequence of Proposition~\ref{prop:strong_confluence} which shall be useful to us.
\begin{proposition}(\cite[Proposition~5]{BK25})
\label{prop:uniqueness_of_geodesics}
Fix $\gamma \in (0,2)$ and consider a whole-plane GFF $h$ and the associated LQG metric $D_h$. Almost surely, simultaneously for all $D_h$-geodesics $P: [0,T] \to \CC$ and $0<s<t<T$, $P|_{[s,t]}$ is the unique $D_h$-geodesic between $P(s)$ and $P(t)$.  
\end{proposition}

\section{Proof of the main result}

In this section, we will complete the proof of Theorem~\ref{thm:1} by proving Propositions~\ref{prop:lower_bound} and ~\ref{prop:upper_bound}. Throughout the rest of the section, we shall always work with a fixed $\gamma \in (0,2)$ and a whole-plane GFF $h$ normalized such that $h_1(0) = 0$. As explained in Section~\ref{subsec:outline}, we will need that any $D_h$-geodesic emanating at a typical point passes through a significant number of $\scX$s as depicted in Figures~\ref{fig:two-panels}, \ref{fig:chi-on-geod}. Moreover, we will require that some additional conditions hold which guarantee that items~\eqref{it:euclidean_lower}-\eqref{it:max_min_thickness_lower} in the statement of Proposition~\ref{prop:lower_bound} hold for the $D_h$-geodesic passing through the corresponding $\scX$s. This will be achieved via a result (Proposition~\ref{lem:G}) taken from \cite{BK25} and applied to a family of well-chosen events $\{H_r(z)\}_{z\in \CC, r>0}$ introduced later in \eqref{eq:3}.

As a matter of notation, for any family of events $\{H_r(z)\}_{z \in \CC,r>0}$ measurable with respect to $\sigma(h)$, we say that $H_r(z)$ is translation and scale invariant if the occurrence of $H_r(z)$ for the field $h$ is the same as the occurrence of the event $H_1(0)$ for the field $z' \mapsto h(z+rz') - h_r(z)$. The events $H_r(z)$ that we will consider in this section will always be \textbf{translation and scale invariant} in the above sense.

We shall now give a precise definition of the configuration $\scX$ that we are going to use. Recall that for all $z,w \in \CC$, all $D_h$-geodesics from $z$ to $w$ are necessarily simple curves, and so the left and right sides of the geodesics are defined as collections of prime ends (see the notational comments from the introduction).

\begin{definition} \label{def:chi_configuration} 
  For distinct points $u_+,v_+,u_-,v_-\in \CC$, the event $\scX_{u_-,v_-}^{u_+,v_+}$ is said to occur if the following hold.
  \begin{enumerate}
  \item The $D_h$-geodesics $P_+,P_-$ from $u_+$ to $v_+$ and $u_-$ to $v_-$ respectively are unique. Moreover, $P_+\cap P_-$ is non-empty and is a non-trivial simple path, whose (distinct) starting and ending points are $u$ and $v$ respectively.
  \item There is a unique $D_h$-geodesic $\Gamma_{u_+,v_-}$ (resp.\ $\Gamma_{u_-,v_+}$) and this is equal to the concatenation of $\Gamma_{u_+,u},\Gamma_{u,v},\Gamma_{v,v_-}$ (resp.\ $\Gamma_{u_-,u},\Gamma_{u,v},\Gamma_{v,v_+}$).
  \item The geodesics $\Gamma_{u_-,u},\Gamma_{v,v_-}$ (resp.\ $\Gamma_{u_+,u}, \Gamma_{v,v_+}$) lie to the right (resp.\ left) of $P_+$ (resp.\ $P_-$).
  \end{enumerate}
\end{definition}

We refer the reader to Figure~\ref{fig:setup1} for a depiction of the event $\scX^{u_+,v_+}_{u_-,v_-}$. Next, we state the main result from \cite{BK25} that we shall require-- the original result appears as \cite[Propositions 23, 35]{BK25}, but the version we state allows for demanding some additional conditions (encapsulated by a family of events $\{H_r(z)\}_{z\in \CC,r>0}$) on the geodesics $P_+,P_-$ forming the $\scX$, and this general version is discussed in \cite[Remarks 24, 36]{BK25}. As notation, we shall use $\pi(\mathtt{h})$ to denote a field $\mathtt{h}$ viewed modulo an additive constant, in the sense that for any global constant $C$, we have $\pi(\mathtt{h}+C)=\mathtt{h}$.

\begin{figure}
  \centering
  \includegraphics[width=\linewidth]{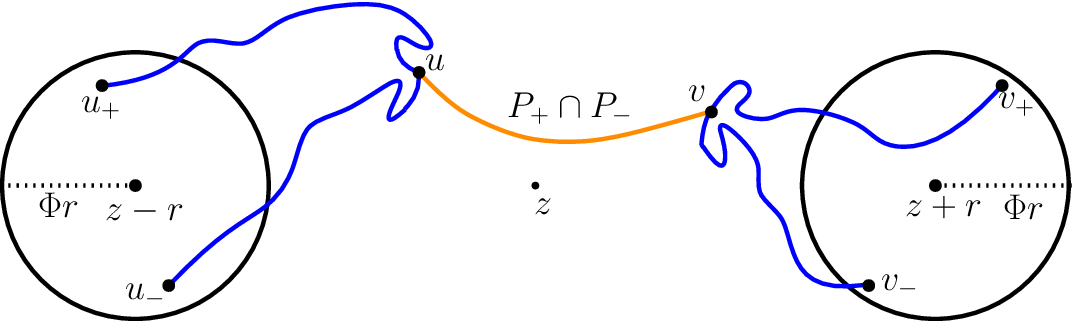}
  \caption{The setup in Proposition~\ref{lem:G}: The blue paths terminating at $u$ represent the geodesics $\Gamma_{u_+,u}$ and $\Gamma_{u_-,u}$, while the blue paths emanating at $v$ represent the geodesics $\Gamma_{v,v_+}$ and $\Gamma_{v,v_-}$ in Definition~\ref{def:chi_configuration}. The orange path represents the intersection of the geodesics $\Gamma_{u_+,v_+}$ and $\Gamma_{u_-,v_-}$ in Definition~\ref{def:chi_configuration}. If the event $\scX^{u_+,v_+}_{u_-,v_-}\cap H_r(z)$ occurs, items~\eqref{eq:H1}-\eqref{eq:H4} hold for the orange path, and so the same is true for any $D_h$-geodesic that contains the orange path. The original source of this figure is \cite[Figure 5]{BK25}.}
  \label{fig:setup1}
\end{figure}

\begin{proposition}
  \label{lem:G}

Fix $\Phi \in (0,1/10)$ and $p_1 \in (0,1)$. There exists a translation and scale invariant family of deterministic points $\{u_-,u_+,v_-,v_+\}_{z \in \CC, r>0}$ satisfying $u_+,u_- \in B_{\Phi r}(z-r)$ and $v_+,v_- \in B_{\Phi r}(z+r)$ such that with $P_+,P_-$ defined by $P_+=\Gamma_{u_+,v_+}, P_-=\Gamma_{v_-,v_+}$, suppose we have a translation and scale invariant family of events $\{H_r(z)\}_{z\in \CC, r>0}$ which are measurable with respect to $\sigma(P_-\cup P_+, \cap_{\varepsilon>0}\sigma( \pi(h)\lvert_{B_{\varepsilon}(P_-\cup P_+)}))$ viewed modulo an additive constant and satisfy $\PP(H_r(z))\geq p_1$, then there exist constants $A > 1, \rho \in (0,1)$ depending only on $p_1$ and constants $c_1,c_2$ depending only on $\gamma$ such the following is true.

Let $U \subseteq \CC$ be a bounded open set. Set $\varepsilon_n = 2^{-n}$ for all $n \in \NN$. With $\mathcal{H}_n = ((\varepsilon_n / 2)^{1/c_1} \ZZ^2) \cap U$, for all $\beta >0$, there exist (random) $0<\varphi < \psi < \beta /20$ depending only on $U$ and $\beta$ such that we have the following for all $n \in \NN$ large enough. For all $\bz \in \mathcal{H}_n$ and all $w \in U$ satisfying $D_h(\bz,w) \geq \beta$, and every $D_h$-geodesic $\Gamma_{\bz,w}$ from $\bz$ to $w$, there exists $(z,r) \in \CC \times [\varepsilon_n^{1/(2c_2)} , \varepsilon_n^{1/(4c_2)}]$ satisfying the following properties.
\begin{enumerate}
\item There exists $z' \in \partial B_r(z)$ such that for the points $u_-,u_+,v_-,v_+$ corresponding to $z'$ and the radius $\rho r/A^2$, the event $\scX^{u_+,v_+}_{u_-,v_-} \cap H_{\rho r/A^2}(z')$ occurs.
    \item With $P_+ = \Gamma_{u_+,v_+}, P_- = \Gamma_{u_-,v_-}$ being the unique $D_h$-geodesics corresponding to the points $u_-,u_+,v_-,v_+$ associated to the point $z'$ and the radius $\rho r/A^2$, we have $P_-\cup P_+\subseteq B_{\rho r/(4A)}(z')$ and further there exists $\varphi<\tau^- < \tau^+<\psi$ such that $P_- \cap P_+ = \Gamma_{\bz,w}|_{[\tau^-,\tau^+]}$.
\item $P_+$ (resp.\ $P_-$) intersects $\Gamma_{\bz,w}$ on its left (resp.\ right) side.
\end{enumerate}

\end{proposition}

\begin{proof}
    The statement of the proposition follows from combining \cite[Propositions 23, 35]{BK25} with \cite[Remarks 24, 36]{BK25}.
\end{proof}

Next, we define the events $H_r(z)$ that we shall use when invoking Proposition~\ref{lem:G}. Let $\delta>0,\alpha \in [-2,2]$ and fix $z \in \CC, r>0$. For now, we work with an unspecified $\Phi \in (0,1/10)$, and this will be chosen later in Lemma~\ref{lem:1} below in a way that depends only on $\delta$ and $\alpha$. Thereafter,  $A>1$ and $\rho \in (0,1)$ shall be the constants chosen in the statement of Proposition~\ref{lem:G} corresponding to the above choice of $\Phi$ and with $p_1 = 1/2$. Let also $P_+$ (resp.\ $P_-$) be the (almost surely unique) $D_h$-geodesic from $u_+$ to $v_+$ (resp.\ $u_-$ to $v_-$), where $u_+,u_-,v_+,v_-$ are the points corresponding to the point $z$ and radius $r$ as in the statement of Proposition~\ref{lem:G}. With $\Delta^{\mathrm{Euc}},\Delta_{\alpha}^{\mathrm{Euc}},\Delta_{\alpha}^{\mathrm{LQG}},\chi,\chi',\alpha_{\mathrm{min}}$, and $\alpha_{\mathrm{max}}$ being the deterministic constants as in the statement of Proposition~\ref{prop:zero_one_law}, consider the events $H_{r,i}(z)$ for $i=1,2,3,4$, defined by 
\begin{align}
  &H_{r,1}(z)=\{\dim_{\mathrm{Euc}} (P_+\cap P_-)\geq \Delta^{\mathrm{Euc}}-\delta\}, \label{eq:H1}\\
 &H_{r,2}(z)=\{\dim_{\mathrm{Euc}}((P_+\cap P_-)\cap \cT_h^{\alpha})\geq  \Delta^{\mathrm{Euc}}_{\alpha}-\delta, \dim_{\mathrm{LQG}}((P_+\cap P_-)\cap \cT^\alpha_h)\geq \Delta^{\mathrm{LQG}}_{\alpha}-\delta\},\label{eq:H2}\\
 &H_{r,3}(z)=\{\chi(P_+\cap P_-)\leq \chi+\delta, \chi'(P_+\cap P_-)\geq \chi'-\delta\},\label{eq:H3}\\
  &H_{r,4}(z)=\left\{\alpha_{\mathrm{max}}(P_+ \cap P_-) \geq \alpha_{\mathrm{max}}-\delta, \alpha_{\mathrm{min}}(P_+ \cap P_-) \leq \alpha_{\mathrm{min}}+\delta\right\}\label{eq:H4}.
\end{align}
Finally, we define
\begin{equation}
  \label{eq:3}
  H_r(z) = \bigcap_{i=1}^4 H_{r,i}(z).
\end{equation}
Note that $H_r(z)$ as defined above is measurable with respect to $\sigma(P_-\cup P_+, \cap_{\varepsilon>0}\sigma( \pi(h)\lvert_{B_{\varepsilon}(P_-\cup P_+)}))$. Indeed, $H_{r,1}(z)$ is clearly measurable with respect to $\sigma(P_1\cap P_+)$. Further, that $H_{r,2}(z)$ and $H_{r,4}(z)$ are measurable with respect to $\sigma(P_-\cup P_+, \cap_{\varepsilon>0} \sigma( \pi(h)\lvert_{B_{\varepsilon}(P_-\cup P_+)}))$ follows by noting that adding a global constant does not affect the thicknesses of points for a GFF. Finally, to obtain the measurability of $H_{r,3}(z)$, we note that conditional on geodesics $P_+, P_-$, replacing the field $h$ in a neighbourhood of $P_+\cup P_-$ by $h+C$ would simply lead to a linear reparametrisation of $P_+,P_-$ when parametrised to have unit speed with respect to the new field, thereby not affecting any H\"older continuity exponents. Furthermore, the translation and scale invariance of the law of $h$ (see Section~\ref{subsec:gff}), the definition of $u_+,u_-,v_+,v_-$ as a translational and scale invariant family of points with respect to $z\in \CC, r>0$, combined with the conformal covariance of the $\gamma$-LQG metric (\cite[Theorem~1.2]{GM21}), implies that the family of events $\{H_r(z)\}_{z \in \mathbb{C}, r>0}$ is scale and translation invariant.

In order to invoke Proposition \ref{lem:G}, we shall show that provided that $\Phi$ therein is chosen to be small enough depending on $\delta$, $H_r(z)$ occurs with probability strictly bounded away from $0$. Before doing so, we state and prove two useful lemmas. Regarding notation, for a path $\eta\colon [0,T]\rightarrow \CC$ and $\varepsilon \in (0, |\eta(0)-\eta(T)|/2)$, we let $\tau_{\varepsilon}^-(\eta)$ be the last time that $\eta$ intersects $\partial B_{\varepsilon}(\eta(0))$ and let $\tau_{\varepsilon}^+(\eta)$ be the first time that $\eta$ intersects $\partial B_{\varepsilon}(\eta(T))$.

\begin{lemma}
  \label{lem:2}
For all $\varepsilon \in (0,1/2)$, we locally set $\tau_{\varepsilon}^- = \tau_{\varepsilon}^-(\Gamma_{-1,1})$ and $\tau_{\varepsilon}^+ = \tau_{\varepsilon}^+(\Gamma_{-1,1})$, %
  Fix $\alpha \in [-2,2]$ and let $\Delta^{\mathrm{Euc}}, \Delta_{\alpha}^{\mathrm{Euc}},\Delta_{\alpha}^{\mathrm{LQG}},\chi,\chi',\alpha_{\mathrm{max}}$, and $\alpha_{\mathrm{min}}$ denote the deterministic constants from the statement of Proposition~\ref{prop:zero_one_law}. Then, almost surely, the following hold as $\varepsilon \to 0$.
  \begin{enumerate}
  \item \label{it:euc_dim_convergence}
  $\dim_{\mathrm{Euc}}(\Gamma_{-1,1}\lvert_{[\tau_\varepsilon^-,\tau_\varepsilon^+]})\rightarrow \Delta^{\mathrm{Euc}}$,
  \item \label{it:euc_dim_thick_points_convergence}
  $\dim_{\mathrm{Euc}}(\Gamma_{-1,1}\lvert_{[\tau_\varepsilon^-,\tau_\varepsilon^+]} \cap \cT^\alpha_h)\rightarrow  \Delta^{\mathrm{Euc}}_{\alpha}$,
  \item \label{it:lqg_dim_thick_points_convergence}
  $\dim_{\mathrm{LQG}}(  \Gamma_{-1,1}\lvert_{[\tau_\varepsilon^-,\tau_\varepsilon^+]}\cap \cT^\alpha_h)\rightarrow \Delta^{\mathrm{LQG}}_{\alpha}$,
  \item \label{it:upper_holder_convergence}
  $\chi(\Gamma_{-1,1}\lvert_{[\tau_\varepsilon^-,\tau_\varepsilon^+]})\rightarrow \chi$,
  \item \label{it:lower_holder_convergence}
  $\chi'(\Gamma_{-1,1}\lvert_{[\tau_\varepsilon^-,\tau_\varepsilon^+]})\rightarrow \chi'$,
  \item \label{it:max_thickness_converegence}
  $\alpha_{\mathrm{max}}(\Gamma_{-1,1}\lvert_{[\tau_\varepsilon^-,\tau_\varepsilon^+]})\rightarrow \alpha_{\mathrm{max}}$,
      \item \label{it:min_thickness_convergence}
      $\alpha_{\mathrm{min}}(\Gamma_{-1,1}\lvert_{[\tau_\varepsilon^-,\tau_\varepsilon^+]})\rightarrow \alpha_{\mathrm{min}}$.
  \end{enumerate}
\end{lemma}

\begin{proof}
To begin, since $\Gamma_{-1,1}$ is almost surely a continuous and simple path, we obtain that
\begin{equation*}
   \Gamma_{-1,1} \setminus \{-1,1\} = \bigcup_{n=1}^{\infty} \Gamma_{-1,1}\lvert_{[\tau_{1/n}^- , \tau_{1/n}^+]},
\end{equation*}
where the union is increasing in $n$. Recall also that in any metric space $(X,d)$, for any increasing sequence of Borel sets $\{U_i\}_{i \in \NN}$, the corresponding Hausdorff dimensions satisfy
\begin{equation}
  \label{eq:incdim}
    \dim(\bigcup_{i=1}^\infty U_i) = \lim_{n \to \infty}  \dim(U_n).
\end{equation}
Now, by Proposition \ref{prop:zero_one_law}, we know that almost surely, $\dim_{\mathrm{Euc}}(\Gamma_{-1,1}) = \Delta^{\mathrm{Euc}}, \dim_{\mathrm{Euc}}(\Gamma_{-1,1} \cap \mathcal{T}_h^{\alpha}) = \Delta_{\alpha}^{\mathrm{Euc}}$ and $\dim_{\mathrm{LQG}}(\Gamma_{-1,1} \cap \mathcal{T}_h^{\alpha}) = \Delta_{\alpha}^{\mathrm{LQG}}$. On combining this with \eqref{eq:incdim}, we immediately obtain items~\eqref{it:euc_dim_convergence},~\eqref{it:euc_dim_thick_points_convergence}, and ~\eqref{it:lqg_dim_thick_points_convergence}.

As for items~\eqref{it:upper_holder_convergence} and ~\eqref{it:lower_holder_convergence}, they follow from combining the fact that almost surely,
\begin{equation*}
    \chi(\Gamma_{-1,1}|_{[\tau_{\varepsilon}^-,\tau_{\varepsilon}^+]}) \to \chi(\Gamma_{-1,1}),\,\,\text{and} \,\, \chi'(\Gamma_{-1,1}|_{[\tau_{\varepsilon}^-,\tau_{\varepsilon}^+]}) \to \chi'(\Gamma_{-1,1}) \,\,\text{as}\,\, \varepsilon \to 0,
\end{equation*}
with the fact that $\chi(\Gamma_{-1,1}) = \chi$ and $\chi'(\Gamma_{-1,1}) = \chi'$ almost surely by Proposition~\ref{prop:zero_one_law}. 

Finally, we prove items~\eqref{it:max_thickness_converegence} and ~\eqref{it:min_thickness_convergence}. Note that almost surely,
\begin{equation*}
    \alpha_{\mathrm{max}}(\Gamma_{-1,1}|_{[\tau_{\varepsilon}^- , \tau_{\varepsilon}^+]}) \to \alpha_{\mathrm{max}}(\Gamma_{-1,1}) \,\, \text{and} \,\, \alpha_{\mathrm{min}}(\Gamma_{-1,1}|_{[\tau_{\varepsilon}^- , \tau_{\varepsilon}^+]}) \to \alpha_{\mathrm{min}}(\Gamma_{-1,1}) \,\, \text{as} \,\, \varepsilon \to 0.
\end{equation*}
Therefore, items~\eqref{it:max_thickness_converegence} and ~\eqref{it:min_thickness_convergence} follow since Proposition~\ref{prop:zero_one_law} implies that $\alpha_{\mathrm{max}}(\Gamma_{-1,1}) = \alpha_{\mathrm{max}}$ and $\alpha_{\mathrm{min}}(\Gamma_{-1,1}) = \alpha_{\mathrm{min}}$ almost surely. 
\end{proof}

\begin{lemma}
  \label{lem:4}
  Fix $\delta>0$ and $\alpha\in [-2,2]$. For every $\nu>0$, there exists a $\Phi\in (0,1/10)$ such that with probability at least $1-\nu$, for all points $x,y\in B_{\Phi}(-1)$, $x',y'\in B_{\Phi}(1)$ and all $D_h$-geodesics $\Gamma_{x,x'},\Gamma_{y,y'}$, we have
    \begin{enumerate}
  \item \label{it:euc_dim_convergence1}
  $\dim_{\mathrm{Euc}}(\Gamma_{x,x'}\cap\Gamma_{y,y'})\geq \Delta^{\mathrm{Euc}}-\delta$,
  \item \label{it:euc_dim_thick_points_convergence1}
  $\dim_{\mathrm{Euc}}( (\Gamma_{x,x'}\cap\Gamma_{y,y'}) \cap \cT^\alpha_h)\geq  \Delta^{\mathrm{Euc}}_{\alpha}-\delta$,
  \item \label{it:lqg_dim_thick_points_convergence1}
  $\dim_{\mathrm{LQG}}( (\Gamma_{x,x'}\cap\Gamma_{y,y'})\cap \cT^\alpha_h)\geq \Delta^{\mathrm{LQG}}_{\alpha}-\delta$,
  \item \label{it:upper_holder_convergence1}
  $\chi(\Gamma_{x,x'}\cap\Gamma_{y,y'})\leq \chi +\delta$,
  \item \label{it:lower_holder_convergence1}
  $\chi'(\Gamma_{x,x'}\cap\Gamma_{y,y'})\geq \chi'-\delta$,
  \item \label{it:max_thickness_converegence1}
  $\alpha_{\mathrm{max}}(\Gamma_{x,x'}\cap\Gamma_{y,y'})\geq \alpha_{\mathrm{max}}-\delta$,
      \item \label{it:min_thickness_convergence1}
      $\alpha_{\mathrm{min}}(\Gamma_{x,x'}\cap\Gamma_{y,y'})\leq \alpha_{\mathrm{min}}+\delta$.
  \end{enumerate}
\end{lemma}
\begin{proof}
  First, we note that as a consequence of Proposition \ref{prop:uniqueness_of_geodesics}, $\Gamma_{x,x'}\cap\Gamma_{y,y'}$, if non-empty, is always a continuous curve. We now take a parameter $\varepsilon>0$ which we shall fix to be small later at the end of the proof, and for now we shall define $\Phi$ in terms of $\varepsilon$. By geodesic confluence (see \cite[Lemma 3.12]{GPS20}\footnote{Alternatively, this can be seen as an immediate consequence of strong confluence (Proposition \ref{prop:strong_confluence}), though this stronger statement is not strictly required here as $-1$ and $1$ are fixed points.}), we can choose $\Phi$ small enough depending on $\nu$ and $\varepsilon$, such that with probability at least $1-\nu/2$, for all geodesics $\Gamma_{x,x'},\Gamma_{y,y'}$ as in the statement of the proposition, we have
  \begin{equation}
    \label{eq:15}
    (\Gamma_{x,x'}\setminus \Gamma_{y,y'}) \cup (\Gamma_{y,y'}\setminus \Gamma_{x,x'})\subseteq B_{\varepsilon}(-1)\cup B_{\varepsilon}(1).
  \end{equation}
  In particular, since $-1\in B_{\Phi}(-1), 1\in B_{\Phi}(1)$, by using the shorthand $\tau_\varepsilon^-=\tau_\varepsilon^-(\Gamma_{-1,1})$ and $\tau_\varepsilon^+=\tau_\varepsilon^+(\Gamma_{-1,1})$, we have
  \begin{equation}
    \label{eq:16}
    \Gamma_{-1,1}\lvert_{[\tau_\varepsilon^-,\tau_\varepsilon^+]}\subseteq \Gamma_{x,x'}\cap \Gamma_{y,y'}.
  \end{equation}

  Thus, in view of \eqref{eq:15}, in order to complete the proof, we need only establish that by choosing $\varepsilon$ small enough, we can ensure that with probability at least $1-\nu/2$, all of the items in the statement of the proposition hold with $\Gamma_{x,x'}\cap\Gamma_{y,y'}$ replaced by $\Gamma_{-1,1}\lvert_{[\tau_\varepsilon^-,\tau_\varepsilon^+]}$. However, this is an immediate consequence of Lemma \ref{lem:2}, and this completes the proof.
\end{proof}

Now we are ready to prove that the events $H_r(z)$ satisfy the conditions required in Proposition \ref{lem:G} provided that $\Phi \in (0,1/10)$ is sufficiently small.

\begin{lemma} \label{lem:1}
Fix $\delta>0,\alpha \in [-2,2]$. Then, there exists a choice of $\Phi \in (0,1/10)$ depending only on $\delta$ and $\alpha$, such that the family of events $\{H_r(z)\}_{z\in \CC,r>0}$ satisfy the assumptions of Proposition~\ref{lem:G} with $p_1 = 1/2$, %
\end{lemma}

\begin{proof}

  We apply Lemma \ref{lem:4} with $\nu=1/2$ and $\delta$ and obtain a resulting $\Phi\in (0,1/10)$. We now use this $\Phi$ with Proposition \ref{lem:G} to obtain the points $u_+,u_-,v_+,v_-$, one set for each $z\in \CC, r>0$, with the set being translation and scale invariant in $z,r$. We recall that these points $u_+,u_-,v_+,v_-$ are in turn used to define the translation and scale invariant family of events $\{H_r(z)\}_{z\in \CC,r>0}$, and our goal now is to verify that this family satisfies $\PP(H_r(z))\geq p_1=1/2$ as is required to apply Proposition \ref{lem:G}. That $H_r(z)\in \sigma(P_-\cup P_+, \cap_{\varepsilon>0}\sigma(\pi(h)\lvert_{B_{\varepsilon}(P_-\cup P_+)}))$ follows by the discussion just after \eqref{eq:3}, so it only remains to obtain the above probability estimate.

 Now, due to the translation and scale invariance of family of events $\{H_r(z)\}_{z\in \CC,r>0}$, it suffices to obtain the probability bound $\PP(H_r(z))\geq p_1=1/2$ for $z=0,r=1$. For this, by our choice, the points $u_-,u_+,v_-,v_+$ satisfy $u_+,u_-\in B_{\phi}(-1)$ and $v_+,v_-\in B_{\phi}(1)$ and thus we can legitimately apply Lemma \ref{lem:4} with $(x,y,x',y')=(u_+,u_-,v_+,v_-)$, and this immediately yields a $\Phi\in (0,1/10)$ for which we have $\PP(H_{1}(0))\geq 1/2$, thereby completing the proof.%

\end{proof}

We are now ready to complete the proof of Theorem~\ref{thm:1}. First, we split up Theorem \ref{thm:1} into proving the following two propositions.

\begin{proposition}\label{prop:lower_bound}
  Fix $\gamma \in (0,2)$ and $\alpha \in [-2,2]$ and let $h$ be a whole-plane GFF. %
Then, with $\Delta^{\mathrm{Euc}}, \Delta_{\alpha}^{\mathrm{Euc}}, \Delta_{\alpha}^{\mathrm{LQG}}, \chi, \chi', \alpha_{\mathrm{min}}$, and $\alpha_{\mathrm{max}}$ being as in Proposition \ref{prop:zero_one_law}, the following holds almost surely. For any $D_h$-geodesic $P\colon [0,T]\rightarrow \CC$ parametrised by $D_h$-length and any $[s,t]\subseteq (0,T)$, we have
\begin{enumerate}
  \item \label{it:euclidean_lower}
  $\dim_{\mathrm{Euc}}(P\lvert_{[s,t]})\geq\Delta^{\mathrm{Euc}}$,
  \item \label{it:thick_points_lower}
  $\dim_{\mathrm{Euc}}(P\lvert_{[s,t]}\cap \cT_h^{\alpha})\geq \Delta^{\mathrm{Euc}}_{\alpha}$ and $\dim_{\mathrm{LQG}}(P\lvert_{[s,t]}\cap \cT^\alpha_h)\geq\Delta^{\mathrm{LQG}}_{\alpha}$,
  \item \label{it:optimal_holder_lower}
  $\chi(P\lvert_{[s,t]})\leq\chi$ and $\chi'(P\lvert_{[s,t]})\geq\chi'$,
  \item \label{it:max_min_thickness_lower}
  $\alpha_{\mathrm{max}}(P|_{[s,t]})\geq\alpha_{\mathrm{max}}$ and $\alpha_{\mathrm{min}}(P|_{[s,t]}) \leq \alpha_{\mathrm{min}}$.
 \end{enumerate}  
\end{proposition}

\begin{proposition}\label{prop:upper_bound}
  Fix $\gamma \in (0,2)$ and $\alpha \in [-2,2]$ and let $h$ be a whole-plane GFF. %
Then, with $\Delta^{\mathrm{Euc}}, \Delta_{\alpha}^{\mathrm{Euc}}, \Delta_{\alpha}^{\mathrm{LQG}}, \chi, \chi', \alpha_{\mathrm{min}}$, and $\alpha_{\mathrm{max}}$ being as in Proposition \ref{prop:zero_one_law}, the following holds almost surely. For any $D_h$-geodesic $P\colon [0,T]\rightarrow \CC$ parametrised by $D_h$-length and any $[s,t]\subseteq (0,T)$, we have
\begin{enumerate}
  \item \label{it:euclidean_upper}
  $\dim_{\mathrm{Euc}}(P\lvert_{[s,t]})\leq\Delta^{\mathrm{Euc}}$,
  \item \label{it:thick_points_lower}
  $\dim_{\mathrm{Euc}}(P\lvert_{[s,t)}\cap \cT_h^{\alpha})\leq \Delta^{\mathrm{Euc}}_{\alpha}$ and $\dim_{\mathrm{LQG}}(P\lvert_{[s,t]}\cap \cT^\alpha_h)\leq\Delta^{\mathrm{LQG}}_{\alpha}$,
  \item \label{it:optimal_holder_lower}
  $\chi(P\lvert_{[s,t]})\geq\chi$ and $\chi'(P\lvert_{[s,t]})\leq\chi'$,
  \item \label{it:max_min_thickness_upper}
  $\alpha_{\mathrm{max}}(P|_{[s,t]}) \leq \alpha_{\mathrm{max}}$ and $\alpha_{\mathrm{min}}(P|_{[s,t]}) \geq \alpha_{\mathrm{min}}$.
 \end{enumerate}  
\end{proposition}

\begin{proof}[Proof of Theorem~\ref{thm:1} assuming Propositions~\ref{prop:lower_bound} and ~\ref{prop:upper_bound}.]
The result follows immediately by combining Propositions~\ref{prop:lower_bound} and ~\ref{prop:upper_bound}.   
\end{proof}

We now start by proving Proposition~\ref{prop:lower_bound}.

\begin{proof}[Proof of Proposition~\ref{prop:lower_bound}]
  Fix $\delta>0,\alpha \in [-2,2]$ and let $\Phi \in (0,1/10)$ and the events $\{H_{r}(z)\}_{z\in \CC, r>0}$ be as obtained by invoking Lemma \ref{lem:1}. Now, by invoking Proposition \ref{lem:G} with the above value of $\Phi$ and $p_1=1/2$, we obtain resulting constants $A>1$ and $\rho\in (0,1)$. For the rest of the proof, we shall work with these constants $\Phi, A,\rho$ and the corresponding events $\{H_r(z)\}_{z\in \CC,r>0}$.

With $[s,t]$ being as in the statement of Proposition \ref{prop:lower_bound}, let $P: [0,T] \to \CC$ be \emph{any} $D_h$-geodesic parametrised according to $D_h$-length and fix $0<s<s'<t'<t<T$. We now let $U$ be an Euclidean ball centered around $0$ with integer radius chosen large enough so as to contain the entire geodesic $P$.  Following the notation in the statement of Proposition~\ref{lem:G}, we define the set $\mathcal{H}_{n}$ using the above $U$ and also define $\varepsilon_n=2^{-n}$ for all $n\in \NN$. Now, let $(\varepsilon_{k_n})$ be a subsequence of $(\varepsilon_n)$ and points $z_n,w_n \in \mathcal{H}_{k_n}$ such that $z_n \to P(s)$ and $w_n \to P(t)$ as $n \to \infty$. Then, by possibly passing into a subsequence, we can assume that there exists a $D_h$-geodesic $\widetilde{P}$ from $P(s)$ to $P(t)$ such that 
\begin{equation*}
    \Gamma_{z_n,w_n} \to \widetilde{P} \,\, \text{as} \,\, n \to \infty
\end{equation*}
in the Hausdorff sense with respect to the Euclidean metric. Note that Proposition~\ref{prop:uniqueness_of_geodesics} implies that it is almost surely the case that $P|_{[s,t]}$ is the unique $D_h$-geodesic between $P(s)$ and $P(t)$, and hence $\widetilde{P} = P|_{[s,t]}$. It follows that
\begin{equation}
  \label{eq:hausdconv}
    \Gamma_{z_n,w_n} \to P|_{[s,t]} \,\, \text{as} \,\, n \to \infty
\end{equation}
in the Hausdorff sense with respect to the Euclidean metric. 

Set $\beta = |t'-s'| > 0$ and let $0<\varphi < \psi < \beta / 20$ be the constants in the statement of Proposition~\ref{lem:G} which correspond to that choice of $\beta$. %
Then, since $z_n,w_n$ converge to $P(s),P(t)$ respectively as $n\rightarrow \infty$, there must exist an $n_0\in \NN$ such that $D_h(z_n,w_n) \geq \beta$ for all $n \geq n_0$. Now, in view of \eqref{eq:hausdconv} and the strong confluence result (Proposition~\ref{prop:strong_confluence}), by possibly taking $n_0$ to be larger, we can assume that
\begin{equation*}
    \Gamma_{z_n,w_n}\lvert_{[\varphi,\psi]} \subseteq P\lvert_{(s,t)}\,\,\text{for all} \,\, n \geq n_0.
\end{equation*}

Further, by possibly taking $n_0$ to be larger and applying Proposition~\ref{lem:G}, we can assume that the following holds for all $n \geq n_0$. There exist $z \in \CC, r>0$, and $\varphi < \tau < \tau' < \psi$, such that if $u_+,u_-,v_+,v_-$ are the points corresponding to the point $z$ at Euclidean scale $\rho r / A^2$ in the statement of Proposition~\ref{lem:G} and $P_+$ (resp.\ $P_-$) is the $D_h$-geodesic from $u_+$ to $v_+$ (resp.\ $u_-$ to $v_-$), then we have $P_- \cap P_+ = \Gamma_{z_n,w_n}\lvert_{[\tau,\tau']}$ and the $D_h$-geodesic $P_- \cap P_+$ satisfies conditions~\eqref{eq:H1}-\eqref{eq:H4} in the definition of the event $H_r'(z)$. In particular, we have that $P_- \cap P_+ \subseteq P\lvert_{(s,t)}$, and thus the following hold.
\begin{enumerate}
    \item $\dim_{\mathrm{Euc}}(P\lvert_{[s,t]}) \geq \dim_{\mathrm{Euc}}(P_- \cap P_+) \geq \Delta^{\mathrm{Euc}} - \delta$,
    \item $\dim_{\mathrm{Euc}}(P\lvert_{[s,t]} \cap \mathcal{T}_h^{\alpha}) \geq \dim_{\mathrm{Euc}}((P_- \cap P_+) \cap \mathcal{T}_h^{\alpha}) \geq \Delta_{\alpha}^{\mathrm{Euc}} - \delta$,
    \item $\dim_{\mathrm{LQG}}(P\lvert_{[s,t]} \cap \mathcal{T}_h^{\alpha}) \geq \dim_{\mathrm{LQG}}((P_- \cap P_+) \cap \mathcal{T}_h^{\alpha}) \geq \Delta_{\alpha}^{\mathrm{LQG}} - \delta$,
    \item $\chi(P|_{[s,t]}) \leq \chi(P_- \cap P_+) \leq \chi + \delta,\,\,\chi'(P|_{[s,t]}) \geq \chi'(P_- \cap P_+) \geq \chi' - \delta$,
    \item 
        $\alpha_{\mathrm{max}}(P|_{[s,t]}) \geq \alpha_{\mathrm{max}}(P_- \cap P_+) \geq \alpha_{\mathrm{max}} - \delta,\,\, \alpha_{\mathrm{min}}(P|_{[s,t]}) \leq \alpha_{\mathrm{min}}(P_- \cap P_+) \leq \alpha_{\mathrm{min}} + \delta$.
\end{enumerate}
Since $\delta>0$ was arbitrary, this completes the proof of the proposition.
\end{proof}

Finally, we now provide the proof of Proposition~\ref{prop:upper_bound}.

\begin{proof}[Proof of Proposition~\ref{prop:upper_bound}]
Fix $\alpha \in [-2,2]$. Then, Proposition~\ref{prop:zero_one_law} implies that it is almost surely the case that for all $z,w \in \QQ^2$ distinct points, we have that the $D_h$-geodesic $\Gamma_{z,w}$ satisfies items~\eqref{it:euclidean_deterministic}-\eqref{it:max_min_thickness_deterministic} in the statement of Proposition~\ref{prop:zero_one_law}. For the rest of the proof, we will assume that we are working on the event that the above holds.

Let $P: [0,T] \to \CC$ be a $D_h$-geodesic and fix $0<s'<s<t<t'<T$. Let also $(z_n), (w_n)$ be sequences of points in $\QQ^2$ such that $z_n \to P(s')$ and $w_n \to P(t')$ as $n \to \infty$. Then, by arguing as in the proof of Proposition~\ref{prop:lower_bound} and using the fact that it is almost surely the case that $P|_{[s',t']}$ is the unique $D_h$-geodesic from $P(s')$ to $P(t')$ by Proposition~\ref{prop:uniqueness_of_geodesics}, we obtain that by possibly passing into a subsequence, we can assume that $\Gamma_{z_n,w_n} \to P|_{[s',t']}$ as $n \to \infty$ in the Hausdorff sense with respect to the Euclidean metric.
Thus, by arguing as in the proof of Proposition~\ref{prop:lower_bound} and using Proposition~\ref{prop:strong_confluence}, we obtain that it is almost surely the case that there exists $n_0 \in \NN$ such that
\begin{equation*}
    P\lvert_{[s,t]} \subseteq \Gamma_{z_n,w_n} \,\, \text{for all} \,\, n \geq n_0.
\end{equation*}
Since $\Gamma_{z_n,w_n}$ satisfies items~\eqref{it:euclidean_deterministic}-\eqref{it:max_min_thickness_deterministic} in the statement of Proposition~\ref{prop:zero_one_law}, it follows that for all $n \geq n_0$, we have the following.
\begin{enumerate}
    \item $\dim_{\mathrm{Euc}}(P\lvert_{[s,t]}) \leq \dim_{\mathrm{Euc}}(\Gamma_{z_n,w_n}) = \Delta^{\mathrm{Euc}}$,
    \item $\dim_{\mathrm{Euc}}(P\lvert_{[s,t]}) \cap \mathcal{T}_h^{\alpha}) \leq \dim_{\mathrm{Euc}}(\Gamma_{z_n,w_n} \cap \mathcal{T}_h^{\alpha}) = \Delta_{\alpha}^{\mathrm{Euc}}$,
     \item $\dim_{\mathrm{LQG}}(P\lvert_{[s,t]} \cap \mathcal{T}_h^{\alpha}) \leq \dim_{\mathrm{LQG}}(\Gamma_{z_n,w_n} \cap \mathcal{T}_h^{\alpha}) = \Delta_{\alpha}^{\mathrm{LQG}}$,
    \item $\chi(P|_{[s,t]}) \geq \chi(\Gamma_{z_n,w_n}) =\chi,\,\,\chi'(P|_{[s,t]}) \leq \chi'(\Gamma_{z_n,w_n}) = \chi'$,
    \item $\alpha_{\mathrm{max}}(P|_{[s,t]}) \leq \alpha_{\mathrm{max}}(\Gamma_{z_n,w_n}) = \alpha_{\mathrm{max}},\,\, \alpha_{\mathrm{min}}(P|_{[s,t]}) \geq \alpha_{\mathrm{min}}(\Gamma_{z_n,w_n}) = \alpha_{\mathrm{min}}$.
\end{enumerate}
This completes the proof of the proposition.
\end{proof}

\printbibliography
 \end{document}